\documentclass[11pt]{amsart}
\usepackage{txfonts}
\usepackage{amsmath,amssymb,amsthm,mathtools}
\usepackage{enumitem}
\usepackage[
  letterpaper,
  left=0.72in,
  right=0.82in,
  top=0.72in,
  bottom=1.08in
]{geometry}

\theoremstyle{plain}

\numberwithin{equation}{section}

\newtheorem{theorem}{Theorem}[section]
\newtheorem{proposition}[theorem]{Proposition}
\newtheorem{lemma}[theorem]{Lemma}

\theoremstyle{remark}
\newtheorem{remark}{Remark}[section]

\newcommand{\dd}{\,\mathrm{d}}
\newcommand{\diver}{\operatorname{div}}

\title[Ball and Spherical-Shell Rigidity] {\textbf{Ball and Spherical-Shell Rigidity from overdetermined Translating solitons}}

\author{Liang Cheng, Li Ma}

\address{Liang Cheng, School of Mathematics and Statistics, and Key Laboratory of Nonlinear Analysis \& Applications (Ministry of Education),\\
Central China Normal University, Wuhan, 430079, P.R. China}
\email{chengliang@ccnu.edu.cn}

\address{Li Ma, School of Mathematics and Physics \\
University of Science and Technology Beijing\\
Xueyuan Road 30, Haidian\\
Beijing 100083 \\
China}
\email{lma17@ustb.edu.cn}

\date{}

\begin{document}

\maketitle

\pagestyle{plain}

\begin{abstract}
We study overdetermined boundary problems for the graphical translating-soliton equation
\[
-\operatorname{div}\!\left(\frac{Du}{\sqrt{1+|Du|^2}}\right)
=\frac{1}{\sqrt{1+|Du|^2}}
\quad\text{in }\Omega,
\qquad
\partial_\nu u=\Gamma H+C
\quad\text{on }\partial\Omega,
\]
where $\Gamma, C$ are constants, $H$ is the mean curvature of the boundary $\partial\Omega$ of the regular bounded domain in $R^n$ such that $H_{\partial B_R}=-1/R$.  For $\Gamma\geq0$, we prove that constant
Dirichlet data force a bounded domain to be a ball.  We also prove a
spherical-shell rigidity theorem for a doubly connected domain with two
ordered boundary heights and $a<u<b$ in the interior.  The argument combines
linearization under reflection, Reichel's critical-plane and annular
continuation principles, curvature comparison, Serrin's corner lemma, and a
radial ODE that excludes the annular alternative in the one-height problem.
Finally, we give explicit counterexamples showing the sharpness of the sign,
ordering, connectedness, and nesting assumptions.
\end{abstract}


\section{Introduction}
Our intention here is to find ball and ring rigidity from the translating soliton equation from the mean curvature flow, widely studied in geometric analysis and nonlinear partial differential equations. We need to impose on overdetermined boundary conditions of the elliptic equations.
Overdetermined elliptic problems prescribe more boundary data than are normally
required for well-posedness. Their solvability therefore imposes strong
restrictions on the geometry of the underlying domain. The classical model is
Serrin's torsion problem: if
\[
-\Delta v=1\quad\text{in }\Omega,
\qquad
v=0,\quad \partial_\nu v=\text{constant}
\quad\text{on }\partial\Omega,
\]
then $\Omega$ is a ball \cite{Serrin1971}. The moving-plane method developed in
that setting, and subsequently extended in many directions
\cite{GidasNiNirenberg1979}, remains one of the principal tools for proving
spherical rigidity. The recent development of the method of moving planes may be found in \cite{CLL}. The equation under study here is the nonlinear translator equation
\begin{equation}\label{eq:translator}
-\diver\left(\frac{Du}{\sqrt{1+|Du|^2}}\right)
=
\frac{1}{\sqrt{1+|Du|^2}}
\quad\text{in }\Omega.
\end{equation}
Its solutions are graphical translating solitons for mean curvature flow, with
the sign convention corresponding to downward unit-speed translation; see, for
example, \cite{HoffmanIlmanenMartinWhite2019}. We assume that the boundary of the domain is smooth. We impose the boundary conditions
\begin{equation}\label{eq:boundary}
 u=a,
 \qquad
 \partial_\nu u=\Gamma H+C
 \quad\text{on }\partial\Omega,
\end{equation}
where $a,\Gamma,C$ are constants and the boundary mean curvature is normalized
so that
\begin{equation}\label{eq:Hconvention}
H=-\frac{1}{n-1}\diver_{\partial\Omega}\nu,
\qquad
H_{\partial B_R}=-\frac1R.
\end{equation}

The second condition in \eqref{eq:boundary} is a curvature-response law that
couples the nonlocal Dirichlet-to-Neumann output of the translator equation to
the local extrinsic geometry of the boundary.  It differs from the classical
Serrin condition in two ways. First, the governing equation is quasilinear. Second,
the normal slope is not constant but is allowed to react linearly to the mean
curvature of the boundary. The sign of this coupling is decisive. With the
convention \eqref{eq:Hconvention}, the moving-plane curvature comparison has the
correct sign when $\Gamma\geq0$. The curvature-contact and orthogonal-corner mechanisms are closely related to those developed by Niebel for an exterior harmonic problem with mean-curvature-dependent Neumann data \cite{NiebelExterior2026}.

\subsection{Main results}
\label{subsec:main-results}

The paper contains two main rigidity theorems. The first concerns a single
Dirichlet height on the whole boundary of the domain $\Omega$ and characterizes balls.  The second
concerns a doubly connected region
$
\Omega=\Omega_{\rm out}\setminus\overline{\Omega_{\rm in}}
$
whose outer and inner boundary components carry two different constant
heights.  In that setting, the strict ordering of the solution between the two
boundary values supplies the sign information required by Reichel's annular
moving-plane theorem \cite[Theorem~2]{Reichel1995}.  We use the topological
continuation part of that theorem as a black box and verify explicitly that the
reflected translator differences satisfy all of its analytic hypotheses.  The
curvature law then replaces the locally constant Neumann data only at the two
terminal critical contacts.  We now state the two rigidity results proved in this paper. 

\begin{theorem}[Spherical rigidity]\label{thm:main}
Let $n\geq2$ and $\alpha\in(0,1)$. Let $\Omega\subset\mathbb{R}^n$ be a bounded
connected domain with $\partial\Omega$ of class $C^{4,\alpha}$; the boundary is
not assumed to be connected. Suppose that
$
u\in C^{4,\alpha}(\overline\Omega)
$
solves
\begin{equation}\label{eq:mainproblem}
\begin{cases}
-\displaystyle\diver\left(\dfrac{Du}{\sqrt{1+|Du|^2}}\right)
=\displaystyle\dfrac{1}{\sqrt{1+|Du|^2}}
&\text{in }\Omega,\\[3mm]
 u=a \ \text{on }\partial\Omega,\ \ \ \ 
 \partial_\nu u=\Gamma H+C&\text{on }\partial\Omega,
\end{cases}
\end{equation}
where $a,C\in\mathbb{R}$ and $\Gamma\geq0$ are constants and $H$ is defined by
\eqref{eq:Hconvention}. Then there exist $x_0\in\mathbb{R}^n$ and $R>0$ such
that
$
\Omega=B_R(x_0).
$
Moreover, $u$ is radial about $x_0$.
\end{theorem}

This result
allows an arbitrary number of boundary components at the outset and shows that
the one-height overdetermined problem nevertheless forces the domain to be a
ball.

\begin{remark}\label{rem:opposite-convention}
If one uses the opposite convention $H_{\partial B_R}=+1/R$, the corresponding
sign condition is $\Gamma\leq0$.
\end{remark}

\begin{remark}
The proof below does not assert rigidity for $\Gamma<0$. In that regime the
curvature comparison at the first contact point has the wrong sign, and the
annular calculation in Section~\ref{sec:annulus} shows that negative curvature
coefficients are naturally compatible with radial shells.
\end{remark}

The second theorem treats a genuinely doubly connected domain whose outer and
inner boundary components carry distinct constant heights.  Its strict
interior ordering is essential: the translator equation excludes interior
minima but does not, by itself, exclude interior maxima.

\begin{theorem}[Spherical-shell rigidity]\label{thm:shell}
Let $n\geq2$ and $\alpha\in(0,1)$. Let $\Omega_{\rm out}$ and
$\Omega_{\rm in}$ be bounded domains of class $C^{4,\alpha}$ with connected
boundaries such that
\[
\overline{\Omega_{\rm in}}\subset\Omega_{\rm out},
\qquad
\Omega=\Omega_{\rm out}\setminus\overline{\Omega_{\rm in}}
\]
is connected. Set
$
\Gamma_1=\partial\Omega_{\rm out},
\ 
\Gamma_2=\partial\Omega_{\rm in},
$
and let $\nu$ denote the outer unit normal of $\Omega$. Suppose that
$u\in C^{4,\alpha}(\overline\Omega)$ solves
\begin{equation}\label{eq:twoheight-problem}
\begin{cases}
-\displaystyle\diver\left(\dfrac{Du}{\sqrt{1+|Du|^2}}\right)
=\displaystyle\dfrac{1}{\sqrt{1+|Du|^2}}
&\text{in }\Omega,\\[3mm]
 u=a&\text{on }\Gamma_1,\\
 u=b&\text{on }\Gamma_2,\\
 \partial_\nu u=c_1H+C_1&\text{on }\Gamma_1,\\
 \partial_\nu u=c_2H+C_2&\text{on }\Gamma_2,
\end{cases}
\end{equation}
where $a,b,C_1,C_2\in\mathbb{R}$, $a<b$, and
\[
c_1\geq0,
\qquad
c_2\geq0.
\]
Assume in addition that
\begin{equation}\label{eq:strict-ordering}
 a<u<b\qquad\text{in }\Omega.
\end{equation}
Then there exist $x_0\in\mathbb{R}^n$ and $0<r<R$ such that
\[
\boxed{
\Omega=B_R(x_0)\setminus\overline{B_r(x_0)}.
}
\]
Moreover, $u$ is radial about $x_0$.
\end{theorem}

The first theorem is sharp with respect to the sign of $\Gamma$: when
$\Gamma<0$, connected annular counterexamples exist.  In the two-height
problem, the ordering $a<u<b$ is not automatic, and the connected nested
outer--inner topology cannot be omitted from the formulation.  Exact examples
illustrating these points are given in Section~\ref{sec:counterexamples}.

A preliminary proof of the results may stop after obtaining full rotational
symmetry, in which case a ball and a concentric annulus are both possible.
One of the purposes of the present paper is to close this gap. We prove directly
from the radial translator ODE that a nontrivial annulus satisfying the same
constant Dirichlet condition on both boundary components necessarily produces
an effective curvature coefficient $\Gamma<0$. Thus no connectedness assumption
on $\partial\Omega$ is needed. Borghini's work provides a recent account of the
corresponding Laplace theory in ring-shaped domains \cite{Borghini2023}.

We also analyze the sharpness of the hypotheses. The same radial ODE yields
actual connected annular counterexamples when the sign condition in the first
theorem is reversed.  For the two-height problem, we separate three issues:
strict ordering, connectedness, and nested topology.  An exact nonmonotone
radial shell shows that the interior ordering is not automatic, while a
disjoint union of two congruent balls disproves the unqualified statement in
which the two boundary hypersurfaces are not required to be the outer and inner
components of a connected region.  The analogous Laplace problem is known to
possess nonradial bifurcating ring domains when the relevant monotonicity sign
fails \cite{KamburovSciaraffia2021}.  More generally, symmetry-breaking
bifurcations for exterior overdetermined problems with nonconstant Neumann data
were obtained by Morabito \cite{Morabito2021}, and non-spherical steady bubbles
and drops occur in the free-boundary Euler setting
\cite{MeyerNiebelSeis2025}.  These results reinforce the need for the sign and
ordering assumptions, but we do not claim an analogous nonradial bifurcation
theorem for the translator operator here.

We now say a few words about the boundary conditions. Curvature-coupled overdetermined conditions arise naturally in free-boundary
models for bubbles and drops with surface tension.  Recent work exhibits both
rigidity and symmetry-breaking phenomena in such models: Meyer, Niebel, and
Seis constructed non-spherical steady bubbles and drops in inviscid fluids,
whereas Niebel established a global rigidity theorem for two-dimensional
bubbles in an appropriate parameter regime
\cite{MeyerNiebelSeis2025,Niebel2025}.  Although the translator problem studied
here is analytically different, these results motivate viewing
\eqref{eq:boundary} as a geometric response law whose rigidity depends
sensitively on the sign and size of the coupling parameter.
The broader literature shows that overdetermined problems on noncompact or
multiply connected domains have a substantially richer geometry than the
bounded simply connected Serrin problem.  Reichel proved radial symmetry for a
large class of elliptic overdetermined problems on exterior domains
\cite{Reichel1997}.  In the plane, Traizet obtained a finite-connectivity
classification through a correspondence with minimal surfaces
\cite{Traizet2014}.  In higher dimensions, Minlend, Weth, and Wu constructed
nontrivial exceptional domains \cite{MinlendWethWu2026}, while Morabito produced
symmetry-breaking bifurcation branches for exterior overdetermined problems
with nonconstant Neumann data \cite{Morabito2021}.  These works explain why the
sign, topology, asymptotic conditions, and precise form of the boundary law
cannot be treated as secondary hypotheses.

\subsection{Organization of the paper}

In Section~\ref{sec:preliminaries} we recall the maximum principles for
linear elliptic equations. In Section~\ref{sec:proof-ball} we prove the
one-height spherical rigidity theorem by moving planes and exclude the annular
alternative through the radial ODE. In Section~\ref{sec:second-theorem} we prove
the two-height spherical-shell theorem by the annular moving-plane method.
Section~\ref{sec:counterexamples} discusses sharpness and gives explicit
counterexamples.

\section{Preliminaries}\label{sec:preliminaries}

\subsection{Geometric, variational, physical, and analytic meaning}
\label{sec:meaning}

Let
\[
M_u=\{(x,u(x)):x\in\Omega\}\subset\mathbb{R}^{n+1},
\qquad
N=\frac{(-Du,1)}{\sqrt{1+|Du|^2}}
\]
be the graph of $u$ with its upward unit normal.  Equation
\eqref{eq:translator} is equivalent to
\[
\mathcal H_{M_u}=\langle-e_{n+1},N\rangle,
\]
so $M_u$ is a translating soliton for mean curvature flow moving with unit
velocity $-e_{n+1}$.  The same equation is the Euler--Lagrange equation of the
weighted area functional
\begin{equation}\label{eq:weighted-area}
\mathcal A[u]=\int_\Omega e^{-u}\sqrt{1+|Du|^2}\,\dd x,
\end{equation}
and hence the graph may also be viewed as a weighted minimal hypersurface.

The Dirichlet condition pins the graphical boundary to one horizontal level;
in the two-height problem the outer and inner boundary components are pinned
to the levels $a$ and $b$.  Since $u$ is constant on each boundary component,
\begin{equation}\label{eq:gradient-normal-boundary}
Du=(\partial_\nu u)\nu
\qquad\text{on }\partial\Omega.
\end{equation}
If $\eta=(\nu,0)$ is the unit normal to the vertical cylinder
$\partial\Omega\times\mathbb R$, then
\[
\langle N,\eta\rangle
=-\frac{\partial_\nu u}{\sqrt{1+(\partial_\nu u)^2}}
=-\frac{\Gamma H+C}{\sqrt{1+(\Gamma H+C)^2}}.
\]
Thus the boundary law prescribes a curvature-dependent contact-angle-type
relation.  The constant $C$ is a uniform slope bias, while $\Gamma$ measures
the response of the boundary inclination to mean curvature.  If $x$ and $u$
have the dimension of length, then $C$ is dimensionless and $\Gamma$ has the
dimension of length.  This law can model curvature-sensitive anchoring,
adhesion, or wetting, but it should not be identified with a unique physical
contact-angle law without a specific constitutive derivation; compare
\cite{Finn1986}.  Since both Dirichlet and Neumann data are imposed, it is an
additional equilibrium constraint rather than the natural boundary condition
of the bulk functional \eqref{eq:weighted-area}.

Analytically, if $\Lambda_\Omega$ denotes the nonlinear
Dirichlet-to-Neumann map for the translator equation, the overdetermined
condition is the local--nonlocal shape equation
\begin{equation}\label{eq:shape-equation}
\Lambda_\Omega(a)-\Gamma H\equiv C.
\end{equation}
Here $\Lambda_\Omega(a)$ depends on the solution throughout the domain,
whereas $H$ is determined by the local second-order geometry of the boundary.
This exceptional balance is the source of rigidity and connects the problem
to the theory of exceptional domains
\cite{Traizet2014,MinlendWethWu2026}.  Integrating the equation also gives the
necessary compatibility identity
\begin{equation}\label{eq:geometric-compatibility}
\int_{\partial\Omega}
\frac{\Gamma H+C}{\sqrt{1+(\Gamma H+C)^2}}\,\dd S
=-\int_\Omega\frac{1}{\sqrt{1+|Du|^2}}\,\dd x<0.
\end{equation}

Finally, the sign of $\Gamma$ has a direct analytic role.  At a first internal
tangency in the moving-plane argument, geometric inclusion gives
$H(P)\geq H(Q)$ under convention \eqref{eq:Hconvention}, while the boundary
law yields
\[
\partial_\nu u(P)-\partial_\nu u(Q)
=\Gamma\bigl(H(P)-H(Q)\bigr).
\]
For $\Gamma\geq0$, this has the sign needed to contradict the Hopf lemma unless
reflection symmetry has already occurred.  Thus the same parameter governs
both the geometric response law and the rigidity mechanism.

\subsection{Ellipticity and maximum principles}

Set
\[
W(p)=\sqrt{1+|p|^2},
\qquad
\Phi(p)=\frac{p}{W(p)},
\qquad
f(p)=\frac1{W(p)}.
\]
Then the equation is
\[
-\diver\Phi(Du)=f(Du).
\]
A direct computation gives
\begin{equation}\label{eq:Dphi}
D\Phi(p)
=
\frac1{W(p)}I-
\frac{p\otimes p}{W(p)^3}.
\end{equation}
For every $\xi\in\mathbb{R}^n$,
\begin{equation}\label{eq:ellipticity}
\frac{|\xi|^2}{W(p)^3}
\leq
D\Phi(p)\xi\cdot\xi
\leq
\frac{|\xi|^2}{W(p)}.
\end{equation}
Since $Du$ is bounded on $\overline\Omega$, the equation is uniformly elliptic
along the given solution.

In nondivergence form, \eqref{eq:translator} becomes
\begin{equation}\label{eq:nondiv}
\left(\delta_{ij}-\frac{u_i u_j}{1+|Du|^2}\right)u_{ij}=-1.
\end{equation}
The coefficient matrix is positive definite. Hence $u$ cannot attain an
interior minimum. Since $u=a$ on $\partial\Omega$, the strong maximum principle
and Hopf lemma yield
\begin{equation}\label{eq:u-positive}
 u>a\quad\text{in }\Omega,
\qquad
 \partial_\nu u<0\quad\text{on }\partial\Omega.
\end{equation}
In particular,
\begin{equation}\label{eq:boundary-sign}
\Gamma H+C<0\quad\text{on }\partial\Omega.
\end{equation}

\subsection{Linearization under reflection}

Fix a direction and rotate coordinates so that it is $e_1$. For
$\lambda\in\mathbb{R}$ define
\[
T_\lambda=\{x_1=\lambda\},
\qquad
x^\lambda=(2\lambda-x_1,x_2,\ldots,x_n).
\]
Whenever reflection across $T_\lambda$ maps the relevant cap into $\Omega$, set
\[
u^\lambda(x)=u(x^\lambda),
\qquad
w_\lambda=u^\lambda-u.
\]
Reflection invariance of the equation gives
\[
-\diver\Phi(Du^\lambda)=f(Du^\lambda).
\]
Subtracting the equation for $u$ and applying the fundamental theorem of
calculus in the gradient variable gives
\begin{equation}\label{eq:linearized-div}
-\diver(A_\lambda Dw_\lambda)-B_\lambda\cdot Dw_\lambda=0,
\end{equation}
where
\begin{align}
A_\lambda(x)
&=
\int_0^1
D\Phi\bigl(Du+t(Du^\lambda-Du)\bigr)\,\dd t,
\label{eq:Adef}\\
B_\lambda(x)
&=
\int_0^1
Df\bigl(Du+t(Du^\lambda-Du)\bigr)\,\dd t.
\label{eq:Bdef}
\end{align}
The matrices $A_\lambda$ are symmetric and uniformly positive definite, by
\eqref{eq:ellipticity}. Since $u\in C^{4,\alpha}(\overline\Omega)$ and $\Phi$ is
smooth, $A_\lambda\in C^{2,\alpha}$ on every closed reflected cap. Multiplying
\eqref{eq:linearized-div} by $-1$ and expanding the divergence gives the
nondivergence equation
\begin{equation}\label{eq:linearized-nondiv}
\mathcal L_\lambda w_\lambda
:= (A_\lambda)_{ij}(w_\lambda)_{ij}
+\bigl(\partial_i(A_\lambda)_{ij}+(B_\lambda)_j\bigr)
(w_\lambda)_j=0.
\end{equation}
There is no zero-order term. Consequently the strong maximum principle and the
Hopf boundary lemma apply to every reflected difference.

We next record three standard moving-plane tools in the precise forms used
below. The first is Serrin's variable-coefficient corner lemma
\cite[Section~4, Lemma~2]{Serrin1971}.

\begin{lemma}[Serrin corner lemma]\label{lem:serrin-corner}
Let $D$ have, near $Q$, a right-angle corner formed by two $C^2$ hypersurfaces
$\{\rho=0\}$ and $\{\sigma=0\}$, with $\rho<0$ and $\sigma<0$ in $D$. Let
\[
Lw=a_{ij}w_{ij}+b_iw_i
\]
be uniformly elliptic, with $a_{ij}=a_{ji}\in C^2$ and $b_i$ bounded. Assume
$w\ge0$ in $D$, $w(Q)=0$, $Lw\le0$, and $w$ is not identically zero. Put
\[
\chi=a_{ij}\rho_i\sigma_j.
\]
If $\chi(Q)>0$, then every direction $s$ entering the corner transversally has
$w_s(Q)>0$. If $\chi(Q)=0$ and every first derivative of $\chi$ tangent to
$\{\rho=0\}\cap\{\sigma=0\}$ vanishes at $Q$, then for every such $s$ either
\[
w_s(Q)>0\qquad\text{or}\qquad w_{ss}(Q)>0.
\]
\end{lemma}

The next two statements are the topological parts of the moving-plane method.
They are independent of the particular terminal Neumann condition.

\begin{lemma}[Critical-plane propagation]\label{lem:propagation}
Suppose a moving-plane sweep in a connected $C^2$ domain reaches a critical
plane $T_{\lambda_*}$, all reflected differences satisfy the maximum principle,
and the reflected difference vanishes identically on one critical cap
component. Then reflection in $T_{\lambda_*}$ preserves the whole domain and
the solution.
\end{lemma}

\begin{proof}
This is the propagation step in the classical moving-plane method; see
Reichel's Step~(I) in \cite{Reichel1997} and the same use in
\cite{NiebelExterior2026}. The point is that the maximal union of reflected
cap components on which equality holds cannot terminate at an interior point,
by uniqueness for the linear equation \eqref{eq:linearized-nondiv}, and cannot
terminate at a new boundary point without producing another critical tangency
or orthogonal contact. Thus the matched union is both relatively open and
relatively closed in the connected domain, and is therefore the whole domain.
\end{proof}

\begin{lemma}[Reichel annular continuation]\label{lem:annular-continuation}
Let
\[
\Omega=\Omega_{\rm out}\setminus\overline{\Omega_{\rm in}}
\]
be a $C^2$ ring domain and let $u$ take constants $a<b$ on the outer and inner
components, respectively, with $a<u<b$ in $\Omega$. Assume that for every
Euclidean reflection the difference of the reflected solution and the original
solution satisfies a linear uniformly elliptic equation with no positive
zero-order term, together with the strong maximum principle and Hopf lemma.
Then, in every direction, Reichel's annular sweep reaches a critical plane at
which exactly one of the following occurs:
\begin{enumerate}[label=\textup{(\roman*)}]
\item the whole ring domain and the solution are symmetric;
\item a reflected piece of one boundary component is internally tangent to the
same component away from the plane, and the reflected difference has a
nontrivial boundary zero there;
\item the plane meets one boundary component orthogonally, and the reflected
difference has a nontrivial corner zero there.
\end{enumerate}
A contact between different boundary components is not terminal.
\end{lemma}

\begin{proof}
This is the specialized form of Reichel's annular moving-plane theorem
\cite[Theorem~2 and its Steps~(I)--(IV)]{Reichel1995}. We emphasize that it is
not the naive maximum-principle argument on the entire geometric cap. Reichel
uses maximal reflected components and a sliding continuation through the
possible interaction with the hole. The strict inequalities $a<u<b$ give the
strict sign whenever reflected and unreflected points lie on different
boundary components, while the Hopf signs at the two components prevent a
terminal loss of comparison at the hole. The proof uses only reflection
invariance, the maximum principle for reflected differences, and these strict
boundary signs; the terminal overdetermined condition is used only after one
of alternatives \textup{(ii)}--\textup{(iii)} has been reached.
\end{proof}

\section{Proof of Theorem~\ref{thm:main}}\label{sec:proof-ball}

\subsection{Moving planes}

Let, for any $\lambda\in R$,
\[
\Sigma_\lambda=\Omega\cap\{x_1<\lambda\}.
\]
Starting with $\lambda$ near the leftmost point of $\Omega$, the reflected cap
is contained in $\Omega$. On the flat part of its boundary,
$
w_\lambda=0\quad\text{on }T_\lambda\cap\Omega.
$
On the curved part, if $x\in\partial\Omega$ and $x^\lambda\in\overline\Omega$,
then
\[
w_\lambda(x)=u(x^\lambda)-a\geq0
\]
by \eqref{eq:u-positive}. The maximum principle applied componentwise to
\eqref{eq:linearized-div} therefore gives
\begin{equation}\label{eq:wpositive}
w_\lambda\geq0
\end{equation}
in every reflected cap before the critical position.

Move the plane until the critical value $\lambda_*$. The standard geometric
alternative is:
\begin{enumerate}[label=\textup{(\roman*)}]
\item a reflected boundary piece is internally tangent to $\partial\Omega$ at a
point away from $T_{\lambda_*}$;
\item $T_{\lambda_*}$ meets $\partial\Omega$ orthogonally.
\end{enumerate}
This formulation also covers contacts between different connected components of
$\partial\Omega$: the reflected cap lies on the interior side of the contacted
boundary, and the outward normals agree at the tangency point.

Write $w=w_{\lambda_*}$. If $w\not\equiv0$ on the relevant cap component, then
$
w>0
$
in its interior by the strong maximum principle. We now exclude both critical
alternatives.

\subsection{Internal tangency}

Let $Q$ be the preimage point in the moving cap and let
$P=Q^{\lambda_*}\in\partial\Omega$ be the tangency point. Since the Dirichlet
value is the same on every boundary component,
we have $
w(Q)=u(P)-u(Q)=0.
$
The Hopf lemma gives
\begin{equation}\label{eq:hopf-tangency}
\partial_{\nu(Q)}w(Q)<0.
\end{equation}
Reflection maps $\nu(Q)$ to $\nu(P)$, so the overdetermined boundary condition
yields
\begin{align}
\partial_{\nu(Q)}w(Q)
&=
\partial_\nu u(P)-\partial_\nu u(Q)\notag\\
&=
\Gamma\bigl(H(P)-H(Q)\bigr).
\label{eq:normal-difference}
\end{align}

To compare the curvatures, translate and rotate coordinates so that the
contact point is the origin and the common outer normal is $e_n$. Write the
original and reflected boundary pieces as
\[
x_n=\varphi(x'),
\qquad
x_n=\psi(x'),
\]
with the corresponding domains locally below the graphs. Then
\[
\varphi(0)=\psi(0)=0,
\qquad
\nabla\varphi(0)=\nabla\psi(0)=0.
\]
The reflected cap lies inside $\Omega$, hence
$
\psi\leq\varphi
$
near the origin. Therefore
\[
D^2\varphi(0)-D^2\psi(0)\geq0
\]
as quadratic forms.

With the convention \eqref{eq:Hconvention}, a graph whose domain lies below it
has mean curvature
\begin{equation}\label{eq:graphH}
H=
\frac1{n-1}
\diver\left(\frac{\nabla\varphi}{\sqrt{1+|\nabla\varphi|^2}}\right).
\end{equation}
At a horizontal tangent point,
$
H=\frac1{n-1}\Delta\varphi.
$
Consequently,
\[
H(P)\geq H_{\mathrm{ref}}(P)=H(Q).
\]
Since $\Gamma\geq0$, \eqref{eq:normal-difference} gives
\[
\partial_{\nu(Q)}w(Q)\geq0,
\]
contradicting \eqref{eq:hopf-tangency}. Thus internal tangency is impossible
unless $w\equiv0$.

\subsection{Orthogonal contact and the corner lemma}

Assume that the critical plane meets $\partial\Omega$ orthogonally at $Q$. After
translation and rotation, suppose
\[
Q=0,
\qquad
T_{\lambda_*}=\{x_1=0\},
\qquad
\nu(Q)=e_n.
\]
Near the origin write
\[
\partial\Omega=\{x_n=\varphi(x_1,\xi)\},
\qquad
\xi=(x_2,\ldots,x_{n-1}),
\]
and
\[
\Omega=\{x_n<\varphi(x_1,\xi)\},
\qquad
\varphi(0)=0,
\qquad
\nabla\varphi(0)=0.
\]
The reflected-cap inclusion implies, for $x_1<0$,
\begin{equation}\label{eq:graph-reflection}
\varphi(x_1,\xi)\leq\varphi(-x_1,\xi).
\end{equation}
The symmetric difference quotient gives
\[
\varphi_1(0,\xi)\geq0.
\]
Orthogonality implies $\varphi_1(0,0)=0$, so
$\xi\mapsto\varphi_1(0,\xi)$ has a local minimum at the origin. Hence
\begin{equation}\label{eq:mixed-third}
\sum_{j=2}^{n-1}\varphi_{1jj}(0)\geq0.
\end{equation}

Let
\[
g(t)=\varphi(t,0)-\varphi(-t,0).
\]
For $t<0$, \eqref{eq:graph-reflection} gives $g(t)\leq0$. Since
$\varphi_1(0)=0$,
\[
g(t)=\frac13\varphi_{111}(0)t^3+o(t^3),
\]
and therefore
\begin{equation}\label{eq:phi111}
\varphi_{111}(0)\geq0.
\end{equation}
Differentiating \eqref{eq:graphH} at the origin gives
\begin{equation}\label{eq:H1}
H_1(0)
=
\frac1{n-1}
\left(
\varphi_{111}(0)+
\sum_{j=2}^{n-1}\varphi_{1jj}(0)
\right)
\geq0.
\end{equation}

Define the boundary normal derivative in graph coordinates by
\[
q(x_1,\xi)
=
\partial_\nu u\bigl(x_1,\xi,\varphi(x_1,\xi)\bigr).
\]
The boundary law gives
$
q=\Gamma H+C,
$
and hence
\begin{equation}\label{eq:q1}
q_1(0)=\Gamma H_1(0)\geq0.
\end{equation}
Because $u=a$ on the graph,
\[
u(x_1,\xi,\varphi(x_1,\xi))=a.
\]
At the origin, all tangential first derivatives of $u$ vanish. Moreover,
\[
q=
\frac{u_n-\sum_{j=1}^{n-1}u_j\varphi_j}
{\sqrt{1+|\nabla\varphi|^2}}.
\]
Differentiating this identity at $0$ and using
$u_j(0)=\varphi_j(0)=0$ for $j<n$, we obtain
\begin{equation}\label{eq:q1un1}
q_1(0)=u_{n1}(0).
\end{equation}
Thus
\begin{equation}\label{eq:un1}
u_{n1}(0)\geq0.
\end{equation}

At the critical position,
\[
w(x)=u(-x_1,x_2,\ldots,x_n)-u(x_1,x_2,\ldots,x_n).
\]
At $0$ one has
\[
\nabla w(0)=0,
\qquad
w_{11}(0)=w_{nn}(0)=0,
\qquad
w_{1n}(0)=-2u_{1n}(0)\leq0.
\]
The local cap lies in the corner $x_1<0$, $x_n<0$. For its inward bisector
\[
s=\frac{-e_1-e_n}{\sqrt2},
\]
we therefore have
\begin{equation}\label{eq:corner-derivatives}
w_s(0)=0,
\qquad
w_{ss}(0)=w_{1n}(0)\leq0.
\end{equation}

We now verify every structural hypothesis in
Lemma~\ref{lem:serrin-corner}. Let $R=\operatorname{diag}(-1,1,\ldots,1)$ be
the reflection matrix. For $x\in T_{\lambda_*}$ one has
\[
Du^{\lambda_*}(x)=RDu(x).
\]
Writing $p=Du(x)$, the vector on the integration segment in
\eqref{eq:Adef} is
\[
z_t=((1-2t)p_1,p_2,\ldots,p_n).
\]
For $j>1$, formula \eqref{eq:Dphi} gives
\[
D\Phi(z_t)_{1j}=-\frac{(1-2t)p_1p_j}{(1+|z_t|^2)^{3/2}}.
\]
The integrand is odd under $t\mapsto1-t$, hence
\begin{equation}\label{eq:cross-coeff-plane}
(A_{\lambda_*})_{1j}=0
\qquad\text{on }T_{\lambda_*},\quad j=2,\ldots,n.
\end{equation}
In particular all derivatives of these functions tangent to the plane vanish.

Choose the corner defining functions
\[
\rho(x)=x_1,
\qquad
\sigma(x)=x_n-\varphi(x_1,\xi),
\]
so that $\rho<0$ and $\sigma<0$ in the local cap. At $Q=0$,
\[
\chi:= (A_{\lambda_*})_{ij}\rho_i\sigma_j
=(A_{\lambda_*})_{1n}
-\sum_{j=1}^{n-1}(A_{\lambda_*})_{1j}\varphi_j=0,
\]
by \eqref{eq:cross-coeff-plane} and $\nabla\varphi(0)=0$. The tangent space to
the corner edge is spanned at $Q$ by $e_2,\ldots,e_{n-1}$. For
$k=2,\ldots,n-1$, differentiation gives
\[
\partial_k\chi(Q)
=\partial_k(A_{\lambda_*})_{1n}(Q)
-(A_{\lambda_*})_{11}(Q)\varphi_{1k}(Q)=0.
\]
Here the first term vanishes because $(A_{\lambda_*})_{1n}$ vanishes
identically on the reflecting plane, and the second vanishes because
$\xi\mapsto\varphi_1(0,\xi)$ has an extremum at $\xi=0$. Thus the equality
case of Lemma~\ref{lem:serrin-corner} applies. It forces either $w_s(0)>0$ or
$w_{ss}(0)>0$, contradicting \eqref{eq:corner-derivatives}. Therefore
orthogonal contact is impossible unless $w\equiv0$.

\subsection{Rotational symmetry}

For every direction, the comparison function at the critical plane vanishes
identically on a critical cap component. Lemma~\ref{lem:propagation} upgrades
this local equality to reflection symmetry of the whole connected domain and
of $u$. Every such reflection fixes the barycenter
\[
x_0=\frac1{|\Omega|}\int_\Omega x\,\dd x,
\]
so all symmetry planes pass through $x_0$. Reflections through hyperplanes
through $x_0$ generate the full orthogonal group about $x_0$. Consequently,
$\Omega$ and $u$ are rotationally invariant about $x_0$.

Because $\Omega$ is connected, there exist $0\leq r<R$ such that
\begin{equation}\label{eq:ball-or-annulus}
\Omega=\{x\in\mathbb{R}^n:r<|x-x_0|<R\}.
\end{equation}
If $r=0$, the theorem is proved. It remains to exclude $r>0$.

\subsection{Exclusion of a concentric annulus}\label{sec:annulus}

Assume that
\[
\Omega=\{x:r<|x|<R\},
\qquad 0<r<R,
\]
and write $u(x)=U(\rho)$, where $\rho=|x|$. Set
\[
p(\rho)=U'(\rho).
\]
The radial form of \eqref{eq:translator} is
\begin{equation}\label{eq:radial-div}
-\frac1{\rho^{n-1}}
\frac{\dd}{\dd\rho}
\left(
\rho^{n-1}\frac{p}{\sqrt{1+p^2}}
\right)
=
\frac1{\sqrt{1+p^2}}.
\end{equation}
Equivalently,
\begin{equation}\label{eq:pODE}
p'
=
-(1+p^2)
\left(1+\frac{n-1}{\rho}p\right).
\end{equation}
Since $U(r)=U(R)=a$ and $U>a$ in the annulus, Hopf's lemma gives
\begin{equation}\label{eq:pboundary}
p(r)>0,
\qquad
p(R)<0.
\end{equation}

\begin{lemma}\label{lem:pmonotone}
The function $p$ is strictly decreasing on $[r,R]$ and has a unique zero
$c\in(r,R)$.
\end{lemma}

\begin{proof}
Define
\[
h(\rho)=p(\rho)+\frac{\rho}{n-1}.
\]
Initially $h(r)>0$. If $h$ had a first zero at $\rho_0$, then
$p(\rho_0)=-\rho_0/(n-1)$, so \eqref{eq:pODE} would give $p'(\rho_0)=0$ and
therefore
\[
h'(\rho_0)=\frac1{n-1}>0.
\]
This is incompatible with a first crossing from positive values. Hence
\[
1+\frac{n-1}{\rho}p(\rho)>0
\]
throughout the interval. Equation \eqref{eq:pODE} now gives $p'<0$. The unique
zero follows from \eqref{eq:pboundary}.
\end{proof}

Let
\[
\alpha=p(r)>0,
\qquad
\beta=-p(R)>0.
\]
For $s\in[0,\alpha]$, let $\rho_+(s)\in[r,c]$ be defined by
\[
p(\rho_+(s))=s,
\]
and for $s\in[0,\beta]$, let $\rho_-(s)\in[c,R]$ be defined by
\[
p(\rho_-(s))=-s.
\]
By \eqref{eq:pODE},
\begin{align}
-\rho_+'(s)
&=
\frac1{(1+s^2)
\left(1+\dfrac{n-1}{\rho_+(s)}s\right)},
\label{eq:rho-plus}\\
\rho_-'(s)
&=
\frac1{(1+s^2)
\left(1-\dfrac{n-1}{\rho_-(s)}s\right)}.
\label{eq:rho-minus}
\end{align}
The denominator in \eqref{eq:rho-minus} is positive by
Lemma~\ref{lem:pmonotone}. Consequently, for every positive $s$ in the common
range,
\begin{equation}\label{eq:density-compare}
\frac{s}{(1+s^2)
\left(1+\dfrac{n-1}{\rho_+(s)}s\right)}
<
\frac{s}{1+s^2}
<
\frac{s}{(1+s^2)
\left(1-\dfrac{n-1}{\rho_-(s)}s\right)}.
\end{equation}

Since the two boundary values are equal,
\[
0=U(R)-U(r)=\int_r^R p(\rho)\,\dd\rho.
\]
Thus the positive and negative signed areas of $p$ agree:
\begin{align}
\int_r^c p(\rho)\,\dd\rho
&=
\int_0^\alpha
\frac{s\,\dd s}{(1+s^2)
\left(1+\dfrac{n-1}{\rho_+(s)}s\right)},
\label{eq:left-area}\\
-\int_c^R p(\rho)\,\dd\rho
&=
\int_0^\beta
\frac{s\,\dd s}{(1+s^2)
\left(1-\dfrac{n-1}{\rho_-(s)}s\right)}.
\label{eq:right-area}
\end{align}
If $\alpha\leq\beta$, the pointwise strict inequality
\eqref{eq:density-compare} would imply that the right-hand side of
\eqref{eq:right-area} is strictly larger than the right-hand side of
\eqref{eq:left-area}, contradicting equality of the two areas. Therefore
\begin{equation}\label{eq:alpha-beta}
\alpha>\beta.
\end{equation}

The outer normal derivative of the annulus is
\[
q_{\mathrm{in}}=-p(r)=-\alpha
\quad\text{on the inner sphere},
\qquad
q_{\mathrm{out}}=p(R)=-\beta
\quad\text{on the outer sphere}.
\]
Thus \eqref{eq:alpha-beta} gives
\begin{equation}\label{eq:q-annulus-negative}
q_{\mathrm{in}}-q_{\mathrm{out}}
=
\beta-\alpha<0.
\end{equation}
On the other hand, with the convention \eqref{eq:Hconvention},
\[
H_{\mathrm{in}}=\frac1r,
\qquad
H_{\mathrm{out}}=-\frac1R.
\]
The boundary law therefore gives
\begin{equation}\label{eq:q-annulus-law}
q_{\mathrm{in}}-q_{\mathrm{out}}
=
\Gamma\left(\frac1r+\frac1R\right)\geq0,
\end{equation}
contradicting \eqref{eq:q-annulus-negative}. Hence $r=0$, and
$
\Omega=B_R(x_0).
$
For the resulting radial solution, write $u(x)=U(\rho)$ with
$\rho=|x-x_0|$ and $p=U'$. Then $p$ satisfies \eqref{eq:pODE}, together with
the regularity condition $p(0)=0$. Since
$H_{\partial B_R}=-1/R$, the boundary law yields the necessary compatibility
relation
\begin{equation}\label{eq:ball-compatibility}
 p(R)=C-\frac{\Gamma}{R}<0.
\end{equation}
Thus the constants $C$, $\Gamma$, and the radius $R$ cannot be prescribed
independently: the right-hand side must equal the boundary slope of the regular
radial translator profile. This completes the proof of
Theorem~\ref{thm:main}.

\section{Proof of Theorem~\ref{thm:shell}}\label{sec:second-theorem}

\subsection{Boundary signs and the annular moving-plane alternative}

By \eqref{eq:strict-ordering} and the Hopf boundary lemma,
\begin{equation}\label{eq:twoheight-hopf}
\partial_\nu u<0\quad\text{on }\Gamma_1,
\qquad
\partial_\nu u>0\quad\text{on }\Gamma_2.
\end{equation}
Notice that on $\Gamma_2$ the normal $\nu$ points into the hole.

The reflected difference satisfies \eqref{eq:linearized-nondiv}. We apply
Lemma~\ref{lem:annular-continuation}. Its analytic assumptions are verified as
follows: reflection invariance is built into the translator operator;
$A_\lambda$ is uniformly positive definite by \eqref{eq:ellipticity}; the
linearized equation has no zero-order term; and the strong maximum principle
and Hopf lemma apply. The strict ordering \eqref{eq:strict-ordering} gives the
required strict sign when a reflected point and an unreflected point belong to
different boundary components, while \eqref{eq:twoheight-hopf} supplies the
boundary signs used in Reichel's sliding continuation. Therefore, unless
symmetry has already occurred, the terminal critical zero is either a
same-component internal tangency away from the plane or an orthogonal contact
with one of the components.

\subsection{Exclusion of same-component internal tangency}

Fix a direction and let $w=u^{\lambda_*}-u$ be the reflected difference at the
critical plane. Suppose that the reflected copy of $\Gamma_i$ is internally
tangent to $\Gamma_i$ at $P$, and let $Q$ be its reflected preimage. Since
$P,Q\in\Gamma_i$, the Dirichlet values agree and
\[
w(Q)=0.
\]
If $w$ is not identically zero, the Hopf lemma yields
\[
\partial_{\nu(Q)}w(Q)<0.
\]
Reflection maps $\nu(Q)$ to $\nu(P)$. On the component $\Gamma_i$, the
boundary condition is
\[
\partial_\nu u=c_iH+C_i.
\]
Consequently, the additive constant $C_i$ cancels and
\begin{equation}\label{eq:twoheight-normal-difference}
\partial_{\nu(Q)}w(Q)
=
\partial_\nu u(P)-\partial_\nu u(Q)
=
c_i\bigl(H(P)-H(Q)\bigr).
\end{equation}
The local graph comparison used in the proof of
Theorem~\ref{thm:main} is independent of the value of the Dirichlet constant
and of the boundary component. It gives
\[
H(P)\geq H(Q).
\]
Since $c_i\geq0$, the right-hand side of
\eqref{eq:twoheight-normal-difference} is nonnegative, contradicting the Hopf
lemma. Thus same-component internal tangency is impossible unless
$w\equiv0$.

\subsection{Exclusion of orthogonal contact}

Suppose next that the critical plane is orthogonal to $\Gamma_i$ at $Q$. The
local calculation in the section on orthogonal contact applies to this
component because it uses only the following facts:

\begin{enumerate}[label=\textup{(\alph*)}]
\item the reflected cap lies on the interior side of the contacted boundary;
\item $u$ is constant on the particular component $\Gamma_i$;
\item on that component, the boundary law is
$\partial_\nu u=c_iH+C_i$ with $c_i\geq0$.
\end{enumerate}

After choosing coordinates with $Q=0$, critical plane $\{x_1=0\}$, and
$\nu(Q)=e_n$, the graph comparison gives
\[
H_1(Q)\geq0.
\]
Writing $q=\partial_\nu u$ along $\Gamma_i$, differentiation tangentially to
$\Gamma_i$ gives
\[
q_1(Q)=c_iH_1(Q)\geq0,
\]
because $C_i$ is constant on the component. Since $u$ is constant on
$\Gamma_i$, the boundary-coordinate calculation from the proof of
Theorem~\ref{thm:main} yields
\[
q_1(Q)=u_{n1}(Q),
\qquad\nu(Q)=e_n,
\]
and hence
\[
u_{n1}(Q)\geq0.
\]
For the reflected difference,
\[
w_{1n}(Q)=-2u_{1n}(Q)\leq0,
\]
whereas all first derivatives and the two pure second derivatives in the
corner directions vanish. The coefficient verification in
\eqref{eq:cross-coeff-plane} and the subsequent computation of the tangential
derivatives of $\chi$ depend only on the constancy of $u$ on the contacted
component. They therefore apply without change on either $\Gamma_1$ or
$\Gamma_2$. Lemma~\ref{lem:serrin-corner} gives the same contradiction as in
the proof of Theorem~\ref{thm:main}. Hence orthogonal contact is impossible
unless $w\equiv0$.

\subsection{Completion of the proof}

For every direction, Lemma~\ref{lem:annular-continuation} and the two contact
contradictions force equality on a critical cap component;
Lemma~\ref{lem:propagation} then yields an exact symmetry of $\Omega$ and $u$.
Every reflection preserves the outer and inner boundary components separately:
they have different topological roles, and they also carry the distinct
Dirichlet values $a$ and $b$. All symmetry planes fix the barycenter $x_0$ of
$\Omega$, so they pass through $x_0$ and generate the full orthogonal group
about that point. Consequently, both boundary components are round spheres
centered at $x_0$. Since $\Gamma_1$ is outer and $\Gamma_2$ is inner, there
exist $0<r<R$ such that
\[
\Gamma_1=\partial B_R(x_0),
\qquad
\Gamma_2=\partial B_r(x_0),
\]
and hence
\[
\Omega=B_R(x_0)\setminus\overline{B_r(x_0)}.
\]
The same reflections show that $u$ is radial.

For completeness, if $u(x)=U(\rho)$, $\rho=|x-x_0|$, and $p=U'$, then the
boundary conditions reduce to
\[
p(R)=C_1-\frac{c_1}{R}<0,
\qquad
-p(r)=C_2+\frac{c_2}{r}>0.
\]
Thus every solution necessarily satisfies
\[
C_1<\frac{c_1}{R},
\qquad
C_2>-\frac{c_2}{r}.
\]
No relation between the constants on the two different boundary components is
needed in the moving-plane argument: only $c_i\geq0$ is used at a terminal
contact on $\Gamma_i$. This proves Theorem~\ref{thm:shell}.

\begin{remark}\label{rem:twoheight-necessity}
The hypotheses $a<b$ and $a<u<b$ are structural rather than cosmetic. The
translator equation rules out an interior minimum but does not rule out an
interior maximum, so the upper bound $u<b$ does not follow automatically from
the PDE. Without the strict ordering, the standard annular moving-plane
continuation may fail, just as it does in classical ring-shaped Serrin
problems with nonmonotone model profiles.
\end{remark}

\begin{remark}
Under the opposite convention $H_{\partial B_R}=+1/R$, the componentwise sign
assumptions become $c_1\leq0$ and $c_2\leq0$. The corresponding radial
compatibility inequalities are obtained by reversing the signs of the
curvature terms.
\end{remark}

\section{Sharpness, exact counterexamples, and failure modes}
\label{sec:counterexamples}

The hypotheses in the two rigidity theorems play different roles.  Some of
 them are genuinely sharp and can be tested by exact solutions of the radial
 translator ODE.  Others are topological assumptions needed to formulate a
 connected shell problem.  We distinguish these issues carefully.

\subsection{A connected annular counterexample when \bf{$\Gamma<0$}}

The sign condition in Theorem~\ref{thm:main} cannot be deleted.  In fact, the
 radial ODE produces a large family of smooth connected counterexamples.

\begin{proposition}[Negative-coupling annular counterexamples]
\label{prop:negative-gamma-counterexample}
Fix $n\geq2$, an inner radius $r>0$, and a number $\alpha>0$.  Let $p$ be the
solution of
\begin{equation}\label{eq:counterexample-p-ivp}
 p'=-(1+p^2)\left(1+\frac{n-1}{\rho}p\right),
 \qquad p(r)=\alpha.
\end{equation}
Then there exist a unique radius $R>r$, a smooth radial function
$u(x)=U(|x|)$ on
\[
 A_{r,R}=B_R\setminus\overline{B_r},
\]
and constants $a,C\in\mathbb{R}$ and $\Gamma<0$ such that
\begin{equation}\label{eq:negative-gamma-counterexample-problem}
\begin{cases}
-\displaystyle\diver\left(\dfrac{Du}{\sqrt{1+|Du|^2}}\right)
=\displaystyle\dfrac1{\sqrt{1+|Du|^2}}
&\text{in }A_{r,R},\\[3mm]
 u=a&\text{on }\partial A_{r,R},\\
 \partial_\nu u=\Gamma H+C&\text{on }\partial A_{r,R}.
\end{cases}
\end{equation}
Consequently, the ball conclusion in Theorem~\ref{thm:main} is false if the
assumption $\Gamma\geq0$ is removed.
\end{proposition}

\begin{proof}
The argument in Lemma~\ref{lem:pmonotone}, applied on every finite interval on
which the solution exists, shows that
\[
 p(\rho)>-\frac{\rho}{n-1}
 \quad\text{and}\quad
 p'(\rho)<0.
\]
Hence $p$ cannot blow up on a finite interval and exists for every
$\rho\geq r$.  While $p\geq0$, equation
\eqref{eq:counterexample-p-ivp} gives $p'\leq-1$, so $p$ has a unique zero
$c>r$.  Since $p$ is strictly decreasing, it is negative after $c$.
Moreover, $p$ cannot converge to a finite limit as $\rho\to\infty$: if
$p(\rho)\to L\leq0$, then the right-hand side of
\eqref{eq:counterexample-p-ivp} converges to $-(1+L^2)<0$, which contradicts
convergence.  Thus
\[
 p(\rho)\longrightarrow-\infty.
\]

Set
\[
 U(\rho)=a+\int_r^\rho p(s)\,\dd s,
\]
where $a$ is arbitrary.  The function $U$ first increases and then tends to
$-\infty$.  Therefore there is a unique $R>c$ such that
\[
 U(R)=U(r)=a.
\]
Write
\[
 \beta=-p(R)>0.
\]
The area comparison in
\eqref{eq:density-compare}--\eqref{eq:alpha-beta}, which applies because the
endpoint values of $U$ agree, gives
\[
 \alpha>\beta.
\]

The outer normal derivatives of the shell are
\[
 q_{\rm in}=-p(r)=-\alpha,
 \qquad
 q_{\rm out}=p(R)=-\beta,
\]
and the normalized mean curvatures are
\[
 H_{\rm in}=\frac1r,
 \qquad
 H_{\rm out}=-\frac1R.
\]
Define
\begin{equation}\label{eq:negative-gamma-definition}
 \Gamma
 =\frac{q_{\rm in}-q_{\rm out}}
 {\dfrac1r+\dfrac1R}
 =\frac{\beta-\alpha}
 {\dfrac1r+\dfrac1R}<0
\end{equation}
and
\begin{equation}\label{eq:negative-C-definition}
 C=q_{\rm in}-\frac{\Gamma}{r}
   =q_{\rm out}+\frac{\Gamma}{R}.
\end{equation}
Then $q=\Gamma H+C$ on both boundary components.  Equation
\eqref{eq:counterexample-p-ivp} is exactly the radial form of the translator
 equation, so $u(x)=U(|x|)$ satisfies
\eqref{eq:negative-gamma-counterexample-problem}.  Since $r>0$, the domain is
not a ball.
\end{proof}

\begin{remark}
Proposition~\ref{prop:negative-gamma-counterexample} proves sharpness of the
sign condition in the one-height theorem in the strongest elementary sense:
the counterexamples are smooth, bounded, connected, and rotationally
symmetric.  The failure is therefore not caused by irregularity or by a loss
of connectedness, but by the sign of the curvature response itself.
\end{remark}

\subsection{The strict two-height ordering is not automatic}

The condition $a<u<b$ in Theorem~\ref{thm:shell} is not a consequence of the
translator equation and the two boundary values.  The following radial family
shows this directly.  It is not a counterexample to the spherical-shell
conclusion, since the underlying domain is already a shell; rather, it proves
that the ordering hypothesis is logically independent of the remaining data.

\begin{proposition}[Nonmonotone two-height radial solutions]
\label{prop:nonmonotone-shell}
Fix $r>0$ and $\alpha>0$, and let $p$ solve
\eqref{eq:counterexample-p-ivp}.  There exists $R>r$ such that, with
\[
 U(\rho)=\int_r^\rho p(s)\,\dd s,
 \qquad
 b=U(r)=0,
 \qquad
 a=U(R)<0,
\]
the shell $A_{r,R}$ admits a radial solution of the translator equation with
\[
 u=a\quad\text{on }\partial B_R,
 \qquad
 u=b\quad\text{on }\partial B_r,
\]
and with a boundary law
\[
 \partial_\nu u=\Gamma H+C
\]
for constants $\Gamma>0$ and $C\in\mathbb{R}$.  Nevertheless,
$
 u>b
$
near the inner boundary, so the strict ordering $a<u<b$ fails.
\end{proposition}

\begin{proof}
As shown above, $p$ is strictly decreasing and tends to $-\infty$, whereas
$U(\rho)\to-\infty$.  Choose $R$ so large that
\[
 p(R)<-\alpha
 \quad\text{and}\quad
 U(R)<U(r)=0.
\]
Put $\beta=-p(R)>\alpha$.  The boundary normal derivatives are
\[
 q_{\rm in}=-\alpha,
 \qquad
 q_{\rm out}=-\beta.
\]
Define
\begin{equation}\label{eq:positive-gamma-nonmonotone}
 \Gamma
 =\frac{q_{\rm in}-q_{\rm out}}
 {\dfrac1r+\dfrac1R}
 =\frac{\beta-\alpha}
 {\dfrac1r+\dfrac1R}>0
\end{equation}
and
\[
 C=q_{\rm in}-\frac{\Gamma}{r}
  =q_{\rm out}+\frac{\Gamma}{R}.
\]
Then the curvature-coupled boundary law holds on both spheres.  Since
$p(r)=\alpha>0$, one has
\[
 U(\rho)>U(r)=b
\]
for $\rho>r$ sufficiently close to $r$.  Hence the upper inequality in
$a<u<b$ fails.
\end{proof}

\subsection{A disconnected counterexample to the unqualified two-component statement}

The original informal statement ``$\partial\Omega=\Gamma_1\cup\Gamma_2$ with
$\Gamma_1\cap\Gamma_2=\varnothing$'' does not by itself say that $\Omega$ is
connected or that one component surrounds the other.  Without that topology,
the shell conclusion is false even for $\Gamma=0$.

\begin{proposition}[Two disjoint congruent balls]
\label{prop:disconnected-counterexample}
Let $R>0$ and choose $x_1,x_2\in\mathbb{R}^n$ with
$|x_1-x_2|>2R$.  There are constants $a<b$ and $C<0$ and a smooth function
$u$ on
\[
 \Omega=B_R(x_1)\mathbin{\dot\cup}B_R(x_2)
\]
such that
\begin{equation}\label{eq:disconnected-problem}
\begin{cases}
-\displaystyle\diver\left(\dfrac{Du}{\sqrt{1+|Du|^2}}\right)
=\displaystyle\dfrac1{\sqrt{1+|Du|^2}}
&\text{in }\Omega,\\[3mm]
 u=a \ \text{on }\Gamma_1:=\partial B_R(x_1),\ \ \
 u=b&\text{on }\Gamma_2:=\partial B_R(x_2),\\
 \partial_\nu u=C&\text{on }\partial\Omega.
\end{cases}
\end{equation}
Thus \eqref{eq:disconnected-problem} has the form of the two-height problem
with $\Gamma=0\geq0$, but $\Omega$ is not a spherical shell.
\end{proposition}

\begin{proof}
Let $V_R$ be the regular radial translator profile on $B_R$, normalized by
$V_R(R)=0$.  For the regular radial translator profile on a ball,
\[
 V_R'(R)=:q_R<0.
\]
Define
\[
 u(x)=
 \begin{cases}
 V_R(|x-x_1|)+a,&x\in B_R(x_1),\\
 V_R(|x-x_2|)+b,&x\in B_R(x_2).
 \end{cases}
\]
The translator equation is invariant under vertical addition of constants, so
$u$ solves the equation on both connected components.  The two balls have the
same radius and hence the same boundary slope $q_R$.  Taking
\[
 C=q_R,
 \qquad
 \Gamma=0,
\]
gives the required Neumann law on the whole boundary.  The domain is a
disjoint union, not a connected annular region.
\end{proof}

\begin{remark}[What is and is not disproved]
Proposition~\ref{prop:disconnected-counterexample} shows that connectedness and
the nested outer--inner geometry in Theorem~\ref{thm:shell} are indispensable
for the formulation.  Proposition~\ref{prop:nonmonotone-shell} shows that the
strict ordering is not automatic.  It does \emph{not}, however, produce a
connected non-spherical counterexample after only the ordering hypothesis is
removed.  No such counterexample for the present translator boundary law is
claimed here.

For comparison, in the classical Laplace torsion problem, Kamburov and
Sciaraffia constructed smooth nonradial doubly connected domains bifurcating
from annuli when the inner-boundary monotonicity sign is reversed
\cite{KamburovSciaraffia2021}.  Their result concerns a different operator and
different Neumann data, so it is evidence for the structural importance of
monotonicity, not a counterexample to the translator theorems proved in this
paper.
\end{remark}

\medskip
\noindent
{\small{\bf{Acknowledgments:}} {Part of this work was done during our visit to Chern institute of mathematics in Nankai University in the summer of 2026. Li Ma would like to express his gratitude to Prof. W.P.Zhang in Chern institute of mathematics at Nankai University.}}

\end{document}